\documentclass[11pt]{amsart}
\usepackage{amsmath,amssymb,amsthm,graphicx}
\usepackage[usenames]{color}
\usepackage[shortlabels]{enumitem}

\newtheorem{thm}{Theorem}
\newtheorem{lem}[thm]{Lemma}

\newtheorem{prop}[thm]{Proposition}

\theoremstyle{definition}

\newtheorem{rmk}[thm]{Remark}

\newtheorem{ques}[thm]{Question}

\newcommand{\CM}{\mathcal{M}}

\title[The existence of a mini. indecomp. genus-2 LF]
{The existence of an indecomposable minimal genus two Lefschetz fibration} 

\begin{document}

\author{Anar Akhmedov}
\address{School of Mathematics, University of Minnesota, Minneapolis, MN, 55455, USA}
\email{akhmedov@math.umn.edu}

\author[N. Monden]{Naoyuki Monden}
\address{Department of Mathematics, Faculty of Science, Okayama University, Okayama 700-8530, Japan}
\email{n-monden@okayama-u.ac.jp}




\begin{abstract} 
It was shown by Usher that any fiber sum of Lefschetz fibrations over $S^2$ is minimal, which was conjectured by Stipsicz. 
We prove that the converse does not hold by showing that there exists an indecomposable minimal genus-2 Lefschetz fibration. 
\end{abstract}

\maketitle

\section{Introduction} Lefschetz fibrations play an important role in 4-manifold topology.  It was shown by Donaldson that, after some blow-ups, any closed symplectic 4-manifold admit a Lefschetz fibration \cite{D1} over $\mathbb{S}^2$. Conversely, Gompf showed that the total space of a Lefschetz fibration admits a symplectic structure, provided the fibers are non-trivial in homology \cite{GS}, generalizing an earlier work of Thurston in \cite{Th}.

For a closed, connected, oriented smooth $4$-manifold $X$, a smooth map $f:X\to \mathbb{S}^2$ is called a genus-$g$ \textit{Lefschetz fibration} if a regular fiber of $f$ is diffeomorphic to a closed oriented surface $\Sigma_g$ of genus $g$ and for each critical point $p$ and $f(p)$ there are complex local coordinate charts agreeing with the orientations of $X$ and $\mathbb{S}^2$ on which is of the form $f(z_1,z_2) = z_1z_2$. 
We suppose that $f$ is injective on the set of critical points $C$ and \textit{relatively minimal}, i.e., no fiber contains a $(-1)$-sphere. We say that $f$ is \textit{minimal} if its total space $X$ is (symplectically) minimal.

The \textit{fiber sum} is one of important and natural operation to construct new Lefschetz fibrations. 
For $i = 1,2$, let $f_i:X_i \to \mathbb{S}^2$ be two genus-$g$ Lefschetz fibrations. 
We remove a fibered neighborhood of a regular fiber $F_i$ from each fibration and glue the resulting 4-manifolds along their boundaries using a fiber-preserving and orientation-reversing diffeomorphism $\phi:F_1\times \mathbb{S}^1 \to F_2\times \mathbb{S}^1$. The result is a new genus-$g$ Lefschetz fibration $f$ on $X:=X_1\sharp_\phi X_2$ called the \textit{fiber sum} of $f_1$ and $f_2$. A Lefschetz fibration is called \textit{indecomposable} if it can not be expressed as a fiber sum. 

Stipsicz \cite{St2} showed that every Lefschetz fibration with $(-1)$-sections is indecomposable (see also \cite{Sm}). 
Note that Lefschetz fibrations with $(-1)$-sections are nonminimal. 
Under this, Stipsicz conjectured that \textit{if a Lefschetz fibration is decomposable, then it is minimal} (see Conjecture 2.3 \cite{St2})
This was proved by Usher \cite{U} (see also \cite{Sato0,B}). 
Our main result is the following. 
\begin{thm}\label{thm:1}
There exists an indecomposable minimal genus-$2$ Lefschetz fibration over $\mathbb{S}^2$, i.e., the converse of Stipsicz's conjecture is false. 
\end{thm}

\section{Positive factorizations and Proofs}
For a genus-$g$ Lefschetz fibration, any fiber containing a critical point is called \textit{singular fiber}, which is obtained by collapsing a simple closed curve, called the \textit{vanishing cycle}, in the regular fiber to a point. 
We call a singular fiber \textit{separating} (resp. \textit{nonseparating}) if the corresponding vanishing cycle is separating (resp. nonseparating) curve on the regular fiber.

Let $\CM_g$ be the mapping class group of $\Sigma_g$, which is the group of isotopy classes of orientation-preserving diffeomorphisms of $\Sigma_g$. 
A genus-$g$ Lefschetz fibration over $\mathbb{S}^2$ is determined by its monodromy representation $\pi_1(\mathbb{S}^2-f(C)) \to \CM_g$, where $C$ is the set of critical points. 
The monodromy of a genus-$g$ Lefschetz fibration $f:X\to \mathbb{S}^2$ comprises a factorization of $\mathrm{id} \in \CM_g$, called a \textit{positive factorization}, as \[t_{v_1}t_{v_2}\cdots t_{v_m} = \mathrm{id},\] where $v_1,\ldots,v_m$ are the vanishing cycles of the singular fibers and $t_{v_i}$ is the right handed Dehn twist along $v_i$. 
Conversely, the above positive factorization in $\CM_g$ gives a genus-$g$ Lefschetz fibration over $\mathbb{S}^2$ with vanishing cycles $v_1,\ldots,v_m$. 

In this article, we focus on genus-$2$ Lefschetz fibrations over $\mathbb{S}^2$. 
For abbreviation, a genus-$2$ Lefschetz fibration $f:X\to \mathbb{S}^2$ is called \textit{of type} $(n,s)$ if $f$ has $n$ nonseparating and $s$ separating singular fibers. 
Note that if $f$ of type $(n,s)$ is a fiber sum of $f_1$ of type $(n_1,s_1)$ and $f_2$ of type $(n_2,s_2)$, then we have $(n,s)=(n_1+n_2,s_1+s_2)$. 

\begin{lem}\label{lem:1}
For a genus-$2$ Lefschetz fibration $X\to \mathbb{S}^2$ of type $(n,s)$, the pair $(n,s)$ satisfies the followings:
\begin{itemize}\rm
\item $n+2s \equiv 0 \pmod{10}$ (see Section 5 \cite{Ma}),
\item $2n-5\geq s$,
\item If $n+2s=10$, then $s\geq 2$ (see Remark 5.1 \cite{Sato1}).
\end{itemize}
\end{lem}
\begin{proof}
We only prove the second inequality. 
In Lemma 5 of \cite{BaykurKorkmaz}, it was shown that $2n-s\geq 3$. 
When we set $n+2s=10k$ (see the first equality), we have $20k-5s\geq 3$, so $4k-s\geq 3/5$. 
Since $k$ and $s$ are integers, we get $4k-s\geq 1$ or equivalently $2n-s \geq 5$ (This inequality can also be obtained from Theorem~\ref{thm:3} below and Corollary 9 in \cite{O}). 
\end{proof}

\begin{prop}\label{prop:1}
A genus-$2$ Lefschetz fibration of type $(n,2n-5)$ is indecomposable. 
\end{prop}
\begin{proof}
Suppose that a genus-$2$ Lefschetz fibration $f$ of type $(n,2n-5)$ is a fiber sum of $f_1$ and $f_2$, where $f_i$ is of type $(n_i,s_i)$ for $i=1,2$. 
We see that $(n,2n-5)=(n_1+n_2, s_1+s_2)$. 
By the second inequality in Lemma~\ref{lem:1}, we have $s_i\leq 2n_i-5$. 
This gives 
\begin{align*}
2n-5=s_1+s_2 \leq 2(n_1+n_2)-10 = 2n-10,
\end{align*}
a contradiction. This finishes the proof. 
\end{proof}

The following theorem is a rough version of the result given by Sato \cite{Sato1}. 
\begin{thm}[Theorem 5.1 \cite{Sato1}]\label{thm:2}
Suppose that a genus-$2$ Lefschetz fibration of type $(n,s)$ is non-minimal. Then, the following holds:
\begin{itemize}
\item If $b_2^+>1$, then the possible pairs $(n,s)$ are $(14, 3)$, $(16,2)$, $(28, 1)$, $(30, 0)$ and $(40, 0)$, 
\item If $b_2^+=1$, then $(n,s)$ satisfies either $n+2s=10$ or $n+2s=20$. 
\end{itemize}
\end{thm}

We present a signature formula for genus-$2$ Lefschetz fibrations given by Matsumoto \cite{Ma}, which was generalized by Endo \cite{E} to genus-$g$ hyperelliptic Lefschetz fibrations. 
\begin{thm}[\cite{Ma}]\label{thm:3}
Let $f:X\to  \mathbb{S}^2$ be a genus-$2$ Lefschetz fibration of type $(n,s)$. Then the signature $\sigma(X)$ of $X$ is 
\begin{align*}
\sigma(X) = -(3n+s)/5. 
\end{align*}
\end{thm}

\begin{lem}\label{lem:3}
There is a genus-$2$ Lefschetz fibration of type $(14,13)$. 
\end{lem}
\begin{proof}
Let us consider a genus-$2$ Lefschetz fibration $f$ of type $(4,3)$ with a positive factorization $t_{a_1}t_{a_2}t_{a_3}t_{a_4}t_{a_5}t_{a_6}t_{a_7}=\mathrm{id}$ in $\CM_2$ (The existence of such a fibration and very explicit algebro-geometric construction is given in \cite{Xiao}. Also, the explicit monodromy of a fibration of type $(4,3)$ was presented \cite{BaykurKorkmaz}). 
By applying cyclic permutations, we may assume that $a_1$ is nonseparating. 
From the relation $t_{a_2}t_{a_3}t_{a_4}t_{a_5}t_{a_6}t_{a_7}=t_{a_1}^{-1}$, we obtain the following two positive factorizations 
\begin{align*}
(t_{a_2}t_{a_3}t_{a_4}t_{a_5}t_{a_6}t_{a_7})^2t_{a_1}^2 = \mathrm{id}, \\
t_{a_1}^2(t_{a_2}t_{a_3}t_{a_4}t_{a_5}t_{a_6}t_{a_7})^2 = \mathrm{id}. 
\end{align*}
Since $a_1$ is a nonseparating curve on $\Sigma_2$, there is a nonseparating curve $b_1$ on $\Sigma_2$ disjoint from $a_1$ and a diffeomorphism $\phi$ such that $\phi(a_1) = b_1$. 
Therefore, by simultaneous conjugation to the above second relation by $\phi$, we obtain 
\begin{align*}
t_{b_1}^2(t_{b_2}t_{b_3}t_{b_4}t_{b_5}t_{b_6}t_{b_7})^2 = \mathrm{id}, 
\end{align*}
where $b_i=\phi(a_i)$ for $i=1,2,\ldots,7$. 
From the above arguments, we get the following positive factorization 
\begin{align*}
(t_{a_2}t_{a_3}t_{a_4}t_{a_5}t_{a_6}t_{a_7})^2t_{a_1}^2t_{b_1}^2(t_{b_2}t_{b_3}t_{b_4}t_{b_5}t_{b_6}t_{b_7})^2 = \mathrm{id}. 
\end{align*}
Here, we consider a sphere $S$ with four boundary components $a,b,c,d$. 
By the lantern relation \cite{De,Jo}, there are three simple closed curves $x,y,z$ on $S$ such that
\begin{align*}
t_at_bt_ct_d=t_xt_yt_z
\end{align*}
Since the genus of $\Sigma_2$ is two, and the two nonseparating curves $a_1$ and $b_1$ are disjoint, $S$ can be embedded in $\Sigma_2$ in such a way that $a$ and $b$ are $a_1$, $c$ and $d$ are $b_1$, $x,z$ are nonseparating and $y$ is separating. 
This gives the following positive factorization
\begin{align*}
(t_{a_2}t_{a_3}t_{a_4}t_{a_5}t_{a_6}t_{a_7})^2 t_x t_y t_z (t_{b_2}t_{b_3}t_{b_4}t_{b_5}t_{b_6}t_{b_7})^2 = \mathrm{id}. 
\end{align*}
Since three of $a_2,\ldots,a_7$ (resp. $b_2,\ldots,b_7$) are nonseparating and the rest are separating curves, we obtain a genus-$2$ Lefschetz fibration of type $(14,13)$, and the proof is complete. 
\end{proof}
\begin{rmk}
The operation using the lantern relation in the above proof is called the \textit{lantern substitution}. 
In \cite{EG}, it was shown that the lantern substitution means the rational blowing down process, which was discovered in \cite{FS1}, along a $-4$-sphere. 
This was generalized in \cite{EMVHM}. 
The lantern substitution preserves the minimality of symplectic 4-manifolds (cf \cite{AM}). 
\end{rmk}

\begin{proof}[Proof of Theorem~\ref{thm:1}]
We show that there must exist any indecomposable minimal Lefschetz fibrations of types $(6,7),(8,11),(10,10)$ and $(14,13)$. 

Let us consider a genus-$2$ Lefschetz fibration of type $(14,13)$. 
Such a fibration is guaranteed to exist by Lemma~\ref{lem:3} and minimal from Theorem~\ref{thm:2}. 
If there is an indecomposable one, then it is the required fibration of Theorem~\ref{thm:1}. 
Therefore, we suppose that any genus-$2$ Lefschetz fibrations of type $(14,13)$ are a fiber sum of Lefschetz fibrations of types $(n_1,s_1)$ and $(n_2,s_2)$. 
Since then, by Lemma~\ref{lem:1}, the following pairs $(n,s)$ are not realizable:
\begin{itemize}
\item If $n+2s=10$, then $(n,s)=(0,6),(1,5),(2,4),(8,1),(10,0)$, 
\item If $n+2s=20$, then $(n,s)=(0,10),(2,9),(4,8)$, 
\item If $n+2s=30$, then $(n,s)=(0,15),(2,14),(4,13),(6,12)$, 
\end{itemize}
we see that the possible pairs $(n_i, s_i)$ are the following:
\begin{enumerate}
\item[(1)] $(n_1,s_1)=(6,7)$ and $(n_2,s_2)=(8,6)$, 
\item[(2)] $(n_1,s_1)=(8,11)$ and $(n_2,s_2)=(6,2)$, 
\item[(3)] $(n_1,s_1)=(10,10)$ and $(n_2,s_2)=(4,3)$. 
\end{enumerate}

We first look at the case (1). 
Then, a genus-$2$ Lefschetz fibration of type $(6,7)$ is indecomposable and minimal, and therefore, it is the required one of Theorem~\ref{thm:1}. 
The proof is as follows. 
The indecomposability immediately follows from Proposition~\ref{prop:1}. 
Assume that there is a non-minimal genus-$2$ Lefschetz fibration $f:X\to \mathbb{S}^2$ of type $(6,7)$. 
Then, from Theorem~\ref{thm:2}, we have $b_2^+(X)=1$. 
Since $\sigma(X) = b_2^+(X) - b_2^-(X) = -5$ by Theorem~\ref{thm:3}, we obtain $b_2^-(X)=6$. 
On the other hand, we have $b_2^-(X)\geq 7$ since every separating singular fiber contains a torus of negative self-intersection, and all of them are linearly independent in homology, a contradiction.

Next, we consider the case (2). 
By Theorem~\ref{thm:2} and Proposition~\ref{prop:1} we see that a genus-$2$ Lefschetz fibration of type $(8,11)$ is indecomposable and minimal, and it is the required fibration.

Finally, we deal the case (3). 
The minimality of a Lefschetz fibration of type $(10,10)$ follows from Theorem~\ref{thm:2}. 
If there is an indecomposable one, we obtain the claimed fibration. 
We suppose that every Lefschetz fibration of type $(10,10)$ is a fiber sum of Lefschetz fibrations of types $(n_3,s_3)$ and $(n_4,s_4)$. 
The possible pairs are $(n_3,s_3)=(6,7)$ and $(n_4,s_4)=(4,3)$ from the above non-realizable pairs $(n,s)$. 
From the case (1), we obtain the required fibration of type $(6,7)$.

This finishes the proof. 
\end{proof}
\begin{rmk}
Strictly speaking, for a genus-$g(\geq 2)$ Lefschetz fibration on $X$ over a closed surface with $s$ separating singular fibers, we have $b_2^-(X)\geq s+1$ as follows. 
Every separating singular fibers contains a surface of self-intersection $-1$. 
Since all of the surfaces are linearly independent in $H_2(X)$ and independent of the class of a smooth fiber, which has selfintersection $0$. 
\end{rmk}

The types $(4,3)$, $(6,7)$ and $(12,19)$, which satisfy $s=2n-5$, were constructed in \cite{Xiao}. 
This observation leads to the following geography question for genus-$2$ Lefschetz fibrations of type $(n,s)$. 
\begin{ques}
Given a pair of integers $(n,s)$ satisfying $n,s\geq 0$, $n+2s\equiv0 \pmod{10}$ and $s\leq 2n-5$, is there a genus-$2$ Lefschetz fibration of type $(n,s)$? 
\end{ques}
\begin{rmk}
After writing the first draft of the paper, the second author was informed by Inanc Baykur on September 19, 2018 that he also obtained the similar proof of the minimality and the indecomposability of type (6,7) and that his former student Kai Nakamura has studied the geography of genus two Lefschetz fibrations and produced some examples of new Lefschetz fibrations in his undergraduate thesis. 
We have not seen Nakamura's work, but according to Baykur, Kai seems to give an example of Lefschetz fibrations of type $(10,10)$. 
Lefschetz fibrations of type $(10,10)$ can also be obtained using the methods of 
\cite{AMon,AM}, by applying two lantern substitution to a word obtained from a twisted fiber sum of  Lefschetz fibrations of types $(8,6)$ and $(6,2)$. 
We will present this and other interesting higher Lefschetz fibrations in our preprint \cite{AMS}, which will appear on arXiv in the near future. 
\end{rmk}

It is natural to ask the following question. 
\begin{ques}
How many indecomposable minimal genus-$g$ Lefschetz fibrations do there exist for $g\geq 2$? 
\end{ques}

\begin{rmk}
Stipsicz also conjectured that every indecomposable Lefschetz fibration has a $(-1)$-section (see Conjecture 2.4 \cite{St2}). 
There are nonminimal counterexamples to this conjecture, i.e., indecomposable nonminimal genus-$g$ Lefschetz fibrations with no $(-1)$-sections ($g=2$ \cite{Sato2}, $g=2,3$ \cite{BH} and $g\geq 2$ \cite{BHM}). 
Our result shows that a minimal counterexample exists. 
\end{rmk}

\begin{rmk}
For $g\geq 3$ and $h=1,2$ we find that genus-$g$ Lefschetz fibrations over $\Sigma_h$ constructed in \cite{KO,Hamada,SY} are indecomposable and minimal. 
The minimality follows from the result of \cite{St1}, and the indecomposability follows from the number of singular fibers and the lower bounds on the number of singular fibers of Lefschetz fibrations over $\mathbb{S}^2$ (see \cite{BK}) and $T^2$ (see \cite{SY}). 
In \cite{BKM}, it was shown that a genus-$g$ Lefschetz fibration over $\Sigma_h$ with a ``maximal section" (see \cite{BKM} for the definition) is indecomposable (as a fiber sum of two genus-$g$ Lefschetz fibrations with a section) if $h\geq 1$. 
Note that a maximal section means that a $(-1)$-section for $h=0$. 
If $g\geq 5$, then we can show that the fibrations given in \cite{KO,Hamada,SY} 
has a maximal section from the constructions of \cite{KO,SY} and Theorem 13 and the technique in Section 3.3 of \cite{BKM}. 
On the other hand, our indecomposable minimal genus-$2$ Lefschetz fibration over $\mathbb{S}^2$ does not admit any maximal sections (i.e. (-1)-sections). 
\end{rmk}

\section*{Acknowledgments} We would like to thank I. Baykur, N. Hamada and K. Yasui for their comments on the first version of this paper posted on the arXiv. 
A. Akhmedov was partially supported by Simons Research Fellowship and  Collaboration Grants for Mathematicians by Simons Foundation. N. Monden was supported by Grant-in-Aid for Young Scientists (B) (No. 16K17601), Japan Society for the Promotion of Science.


\begin{thebibliography}{99}






\bibitem{AMon} A. Akhmedov and N. Monden,  \textit{Constructing Lefschetz fibrations via daisy substitution}, Kyoto J. Math, \textbf{56} (2016), no. 3, 501--529.

\bibitem{AM} A, Akhmedov and N. Monden, \textit{Genus two Lefschetz fibrations with $b^{+}_{2}=1$ and ${c_1}^{2}=1$}, arXiv:1509.01853. 

\bibitem{AMS} A. Akhmedov, N. Monden and S. Sakalli, \textit{Lefschetz fibrations via monodromy substitutions}, preprint. 

\bibitem{B} R. I. Baykur, \textit{Minimality and fiber sum decompositions of Lefschetz fibrations}, Proc. Amer. Math. Soc. \textbf{144} (2016), 2275--2284. 

\bibitem{BH} R. I. Baykur and K. Hayano, \textit{Multisections of Lefschetz fibrations and topology of symplectic 4-manifolds}, Geom. Topol. \textbf{20} (2016), no. 4, 2335--2395. 

\bibitem{BHM} R. I. Baykur, K. Hayano and N. Monden, \textit{Unchaining surgery and topology of symplectic 4-manifolds}, in preparation. 

\bibitem{BaykurKorkmaz} R. I. Baykur and M. Korkmaz, \textit{Small Lefschetz fibrations and exotic 4-manifolds}, Math. Ann. \textbf{367} (2017), no. 3-4, 1333--1361.

\bibitem{BK} V. Braungardt and D. Kotschick, \textit{Clustering of critical points in Lefschet fibrations and the symplectic Szpiro inequality}, Trans. Amer. Math. Soc. \textbf{355} no. 8, (2003), 3217--3226.

\bibitem{BKM} R. I. Baykur, M. Korkmaz and N. Monden, \textit{Sections of surface bundles and Lefschetz fibrations}, Trans. Amer. Math. Soc. \textbf{365} (2013),  no. 11, 5999--6016. 



\bibitem{D1} S. Donaldson, \textit{Lefschetz pencils on symplectic manifolds}, J. Differential Geom. \textbf{53}, 205--236, 1999.



\bibitem{De} M. Dehn, \textit{Die Gruppe der Abbildungsklassen}, Acta Math. \textbf{69} (1938), 135--206.

\bibitem{E} H. Endo, \textit{Meyer's signature cocyle and hyperelliptic fibrations}, Math. Ann., \textbf{316} (2000), 237--257.

\bibitem{EG} H. Endo and Y. Gurtas, \textit{Lantern relations and rational blowdowns}, Proc. Amer. Math. Soc. \textbf{138} (2010), no. 3, 1131--1142.

\bibitem{EMVHM} H. Endo, T. E. Mark, and J. Van Horn-Morris, \textit{Monodromy substitutions and rational blowdowns}, J. Topol. \textbf{4} (2011), 227--253.



\bibitem{FS1} R. Fintushel and R. Stern, \textit{Rational blowdowns of smooth $4$-manifolds}, J. Differential Geom. \textbf{46} (1997), 181--235.





\bibitem{GS}  R. E. Gompf and A. I. Stipsicz, \textit{$4$-Manifolds and Kirby Calculus}, Graduate Studies in Mathematics, vol. 20, Amer. Math. Soc., Providence, RI, 1999.


\bibitem{Hamada} N. Hamada, \textit{Upper bounds for the minimal number of singular fibers in a Lefschetz fibration over the torus}, Michigan Math. J. \textbf{63} (2014), no. 2, 275--291. 



\bibitem{Jo} D. Johnson, \textit{Homeomorphisms of a surface which act trivially on homology}, Proc. Amer. Math. Soc. \textbf{75} (1979), 119--125.



\bibitem{KO} M. Korkmaz and B. Ozbagci, \textit{Minimal number of singular fibers in a Lefschetz fibration}, Proc. Amer. Math. Soc. \textbf{129} (2000), no. 5, 1545--1549.



\bibitem{Ma} Y. Matsumoto, \textit{Lefschetz fibrations of genus two --- a topological approach}, Topology and Teichm$\ddot{\textnormal{u}}$ller spaces (Katinkulta, 1995), 123--148, World Sci. Publ., River Edge, NJ, 1996.


\bibitem{O} B. Ozbagci, \textit{Signatures of Lefschetz fibrations}, Pacific J. Math. \textbf{202} (2002), 99--118.




\bibitem{Sato0} Y. Sato, \textit{The Stipsicz's conjecture for genus-2 Lefschetz fibrations}, preprint. 

\bibitem{Sato1} Y. Sato, \textit{2-spheres of square -1 and the geography of genus-2 Lefschetz fibrations}, J. Math. Sci. Univ. Tokyo \textbf{15} (2008), 461--491.

\bibitem{Sato2} Y. Sato, \textit{The necessary condition on the fiber-sum decomposability of genus-2 Lefschetz fibrations}, Osaka J. Math. \textbf{47} (2010) no. 4, 949--963.


\bibitem{Sm1} I. Smith, \textit{Lefschetz fibrations and the Hodge bundle}, Geom. Topol. \textbf{3} (1999), 211--233.

\bibitem{Sm} I. Smith, \textit{Geometric monodromy and the hyperbolic disc}, Quarterly J. Math. \textbf{52} (2001), no. 2, 217--228.

\bibitem{St1} A. Stipsicz, \textit{Chern numbers of certain Lefschetz fibrations}, Proc. Amer. Math. Soc. \textbf{128} (1999), no. 6, 1845--1851.


\bibitem{St2} A. Stipsicz, \textit{Indecomposability of certain Lefschetz fibrations}, Proc. Amer. Math. Soc. \textbf{129} (2001), no. 5, 1499--1502.

\bibitem{SY} A. Stipsicz and K.-H. Yun, \textit{On minimal number of singular fibers in Lefschetz fibrations over the torus}, Proc. Amer. Math. Soc. \textbf{145} (2017), no. 8, 3607--3616.

\bibitem{Th} W. P. Thurston, \textit{Some simple examples of symplectic manifolds}, Proc. Amer. Math. Soc. {\bf 55}, (1976), 467--468. 




\bibitem{U} M. Usher, \textit{Minimality and Symplectic Sums}, Int. Math. Res. Not. (2006), Art ID \textbf{49857}, 17.

\bibitem{Xiao} G. Xiao, \textit{Surfaces fibr$\grave{e}$es en courbes de genre deux}, Lecture Notes in Math, vol.1137. Springer, Berlin (1985).

\end{thebibliography}
\end{document}